\documentclass[reqno]{amsart}

\usepackage{tabu}
\usepackage{amssymb}
\usepackage{mathtools}
\usepackage{a4wide,amsmath}
\usepackage{mathrsfs}
\usepackage{amsthm}
\usepackage{dsfont}
\numberwithin{equation}{section}
\numberwithin{figure}{section}
\numberwithin{table}{section}
\usepackage{bbm}
\usepackage{subfig}
\usepackage{enumerate}
\usepackage{graphicx}		  
\usepackage{ifpdf}
\ifpdf
\DeclareGraphicsExtensions{.pdf,.eps,.jpg,.png}	
\usepackage[suffix=]{epstopdf}
\fi
\usepackage{xcolor}
\usepackage[utf8]{inputenc}
\usepackage{hyperref}
\hypersetup{hidelinks}

\long\def\MSC#1\EndMSC{\def\arg{#1}\ifx\arg\empty\relax\else
	{\narrower\noindent%
		{2020 Mathematics Subject Classification}: #1\\} \fi}
\long\def\PACS#1\EndPACS{\def\arg{#1}\ifx\arg\empty\relax\else
	{\narrower\noindent%
		{PACS numbers}: #1}\fi}
\long\def\KEY#1\EndKEY{\def\arg{#1}\ifx\arg\empty\relax\else
	{\narrower\noindent%
		Keywords: #1\\}\fi}

\theoremstyle{plain}
\newtheorem{theorem}{Theorem}[section]
\newtheorem{lemma}[theorem]{Lemma}
\newtheorem{proposition}[theorem]{Proposition}
\newtheorem{corollary}[theorem]{Corollary}
\theoremstyle{definition}

\newtheorem{assumption}[theorem]{Assumption}
\theoremstyle{remark}
\newtheorem{remark}[theorem]{Remark}

\newcommand{\norm}[1]{\lVert#1\rVert}
\newcommand{\abs}[1]{\lvert#1\rvert} 
\newcommand{\inner}[1]{\langle#1\rangle} 
\newcommand{\redel}{\mathop{\textup{Re}}}

\newcommand{\essinf}{\mathop{\textup{ess\,inf}}}

\newcommand{\spanm}{\mathop{\textup{span}}}

\newcommand{\I}{\mathrm{i}}    
\newcommand{\e}{\mathrm{e}}    
\newcommand{\di}{\mathrm{d}}   

\newcommand{\R}{\mathbb{R}}
\newcommand{\N}{\mathbb{N}}
\newcommand{\C}{\mathbb{C}}
\newcommand{\Z}{\mathbb{Z}}
\usepackage{leftidx}
\newcommand{\F}{\mathop{{}_2 F_1}}

\begin{document}
	\title[Infinite-dimensional Lipschitz stability]{Infinite-dimensional Lipschitz stability in the Calder\'on problem and general Zernike bases}
	
	\author[H.~Garde]{Henrik Garde}
	\address[H.~Garde]{Department of Mathematics, Aarhus University, Ny Munkegade 118, 8000 Aarhus C, Denmark.}
	\email{garde@math.au.dk}
	
	\author[M.~Hirvensalo]{Markus Hirvensalo}
	\address[M.~Hirvensalo]{Department of Mathematics and Systems Analysis, Aalto University, P.O. Box~11100, 00076 Helsinki, Finland.}
	\email{markus.hirvensalo@aalto.fi}
	
	\author[N.~Hyv\"onen]{Nuutti Hyv\"onen}
	\address[N.~Hyv\"onen]{Department of Mathematics and Systems Analysis, Aalto University, P.O. Box~11100, 00076 Helsinki, Finland.}
	\email{nuutti.hyvonen@aalto.fi}
	
	\begin{abstract}
		Calder\'on's inverse conductivity problem has, so far, only been subject to conditional logarithmic stability for infinite-dimensional classes of conductivities and to Lipschitz stability when restricted to finite-dimensional classes. Focusing our attention on the unit ball domain in any spatial dimension $d\geq 2$, we give an elementary proof that there are (infinitely many) infinite-dimensional classes of conductivities for which there is Lipschitz stability. In particular, Lipschitz stability holds for general expansions of conductivities, allowing all angular frequencies but with limited freedom in the radial direction, if the basis coefficients decay fast enough to overcome the growth of the basis functions near the domain boundary. We construct general $d$-dimensional Zernike bases and prove that they provide examples of infinite-dimensional Lipschitz stability.
	\end{abstract}	
	\maketitle
	
	\KEY
	Calder\'on problem, inverse conductivity problem, Lipschitz stability, Zernike basis.
	\EndKEY
	
	\MSC
	35R30, 35R25, 33C45.
	\EndMSC
	
	\section{Introduction} 
	
	Let $B$ be the Euclidean unit ball in $\R^d$ for a spatial dimension $d\in\N\setminus\{1\}$. For a conductivity coefficient $\gamma\in L^\infty(B;\R)$, with $\essinf \gamma > 0$, the conductivity equation is
	\begin{equation} \label{eq:condeq}
		-\nabla\cdot(\gamma\nabla u) = 0 \text{ in } B.
	\end{equation}
	One can consider boundary measurements in the form of the Dirichlet-to-Neumann map
	\begin{equation*}
		\Lambda(\gamma) \colon H^{1/2}(\partial B) \to H^{-1/2}(\partial B),
	\end{equation*}
	given as 
	\begin{equation*}
		u|_{\partial B} \mapsto \nu\cdot(\gamma\nabla u)|_{\partial B}
	\end{equation*}
	for the solutions (electric potentials) $u\in H^1(B)$ to \eqref{eq:condeq}. Calder\'on's inverse conductivity problem is to determine $\gamma$ from $\Lambda(\gamma)$; see, e.g.,~Calder\'on's original paper~\cite{Calderon1980} or the survey papers \cite{Borcea2002a,Uhlmann2009} for more information.
	
	In this paper we assume enough regularity for Lipschitz stability to hold for boundary determination, i.e.\ we consider conductivity coefficients in
	\begin{equation*}
		C_+(\overline{B}) = \{\, \gamma \in C(\overline{B};\R) \mid \min_{x\in\overline{B}}\gamma(x) > 0 \,\}.
	\end{equation*}
	There is the following classical result on stability at $\partial B$.
	\begin{theorem}[Sylvester--Uhlmann \cite{SylvesterUhlmann88}] \label{thm:bdrstab}
		There exists $C_{\partial B} > 0$ such that for any $\gamma_1,\gamma_2\in C_+(\overline{B})$,
		\begin{equation*}
			\norm{\gamma_1-\gamma_2}_{L^\infty(\partial B)} \leq C_{\partial B}\norm{\Lambda(\gamma_1)-\Lambda(\gamma_2)}_{\mathscr{L}(H^{1/2}(\partial B),H^{-1/2}(\partial B))}.
		\end{equation*}
		In particular, $C_{\partial B}$ does not depend on the minima/maxima of $\gamma_1$ and $\gamma_2$.
	\end{theorem}
	We will give infinite-dimensional examples where $\gamma_1-\gamma_2$ can be bounded by its restriction to the boundary, which leads to Lipschitz stability for the Calder\'on problem through the use of Theorem~\ref{thm:bdrstab}.
	
	\subsection{Infinite-dimensional Lipschitz stability} \label{sec:results}
	
	The first result is a fairly simple observation that follows directly from Theorem~\ref{thm:bdrstab} under the weak maximum principle.
	\begin{corollary} \label{cor1}
		Let $\gamma_1,\gamma_2\in C_+(\overline{B})$. If $(\gamma_1-\gamma_2)\in C^2(B)\cap C^1(\overline{B})$ and is a solution to an elliptic equation (with nonnegative zero-order term) satisfying a weak maximum principle, then there is Lipschitz stability in $B$,
		\begin{equation*}
			\norm{\gamma_1-\gamma_2}_{L^\infty(B)} \leq C_{\partial B}\norm{\Lambda(\gamma_1)-\Lambda(\gamma_2)}_{\mathscr{L}(H^{1/2}(\partial B),H^{-1/2}(\partial B))},
		\end{equation*}
		where $C_{\partial B}$ is the constant from Theorem \ref{thm:bdrstab}.
	\end{corollary}
	\begin{remark}
		Although we focus on $B$ in this paper, Corollary~\ref{cor1} obviously generalizes to other bounded smooth domains, as this is also the case for Theorem~\ref{thm:bdrstab}.
	\end{remark}
	
	Since the condition in Corollary~\ref{cor1} is for the difference $\gamma_1-\gamma_2$, the two conductivities can also share a common component not related to a partial differential equation:
	\begin{equation*}
		\gamma_1 = \gamma_0 + \widetilde{\gamma}_1 \quad \text{and} \quad \gamma_2 = \gamma_0 + \widetilde{\gamma}_2,
	\end{equation*}
	where $\gamma_0$ could, e.g., be a large enough $C_+(\overline{B})$ function to ensure that $\gamma_1$ and $\gamma_2$ are positive, while $\widetilde{\gamma}_1$ and $\widetilde{\gamma}_2$ could be solutions to the same equation satisfying a weak maximum principle. A canonical example is that $\widetilde{\gamma}_1$ and $\widetilde{\gamma}_2$ are harmonic functions bounded in $\overline{B}$.
	
	Next we will look for other infinite-dimensional classes for which there is interior Lipschitz stability. For $x\in B\setminus\{0\}$ consider polar coordinates $r = \abs{x}$ and $\theta = \frac{x}{\abs{x}} \in \partial B$ such that $x = r\theta$. The following are the core assumptions on the basis functions employed in this work:
	\begin{assumption}{}\ \label{assump}
		\begin{enumerate}[(i)]
			\item Let $\{f_n\}_n$ be an orthonormal basis for $L^2(\partial B)$.
			\item For each $n$, let $\{g_{n,k}\}_k$ be an orthonormal set in the weighted space
			\begin{equation*}
				L^2_{r^{d-1}}([0,1]) = \bigl\{\, g\colon [0,1]\to\C \text{ measurable} \bigm| \int_0^1 \abs{g(r)}^2r^{d-1}\,\di r < \infty \,\bigr\}.
			\end{equation*}
			\item Assume that $g_{n,k}(r)$ is defined and continuous at $r=1$.
			\item Assume that $\inf_{n,k}\abs{g_{n,k}(1)} > 0$.
			\item Define $\psi_{n,k}(r\theta) = g_{n,k}(r)f_n(\theta)$, and assume there are scalars $a_{n,k}$ such that 
			\begin{equation*}
				\norm{\psi_{n,k}}_{L^\infty(U)} \leq a_{n,k}
			\end{equation*}
			in a spherical shell $U = \{\, x \mid 1-\eta \leq \abs{x}\leq 1 \,\}$ for some arbitrarily small $\eta>0$.
		\end{enumerate}
	\end{assumption}
	\begin{remark}
		As a special case, the functions $g_{n,k}$ can be constant in the index $n$, which corresponds to independently finding orthonormal bases $\{f_n\}_n$ and $\{g_k\}_k$ in the variables $\theta$ and~$r$, respectively. We allow the dependence of $g_{n,k}$ on $n$ because it is natural when the functions $\psi_{n,k}$ originate from an eigenvalue problem, as is the case for the Zernike basis in Section~\ref{sec:zernike}.
	\end{remark}
	Due to Assumptions~\ref{assump}(i) and ~\ref{assump}(ii), $\{\psi_{n,k}\}_{n,k}$ is an orthonormal set in $L^2(B)$. Let 
	\begin{equation*}
		W = \overline{\spanm\{\psi_{n,k}\}_{n,k}},
	\end{equation*}
	with the closure taken in $L^2(B)$, so that $\{\psi_{n,k}\}_{n,k}$ becomes an orthonormal basis for $W$. Expanding elements of $W$ in the basis $\{\psi_{n,k}\}_{n,k}$ naturally involves coefficients $(c_{n,k})\in\ell^2$. However, to ensure that a boundary trace can be taken inside such an expansion (Lemma~\ref{lemma:trace}), we need to impose an additional decay condition on $(c_{n,k})$. To this end, we say that a sequence $(c_{n,k})$ belongs to the weighted $\ell_{a_{n,k}}^1$ space if 
	\begin{equation*}
		\sum_{n,k} \abs{c_{n,k}}a_{n,k} < \infty,
	\end{equation*}
	and we further define 
	\begin{equation*}
		\mathcal{A} = \bigl\{\, f \in W \mid (c_{n,k}) \in \ell^1_{a_{n,k}} \text{ where } c_{n,k}=\inner{f,\psi_{n,k}}_{L^2(B)} \,\bigr\}.
	\end{equation*}
	This leads to our main result on stability.
	
	\begin{theorem} \label{thm:main}
		Let $\gamma_1,\gamma_2\in C_+(\overline{B})$. Assume $(\gamma_1-\gamma_2)\in \mathcal{A}$ and for $c_{n,k} = \inner{\gamma_1-\gamma_2,\psi_{n,k}}_{L^2(B)}$ assume there exists $\epsilon>0$ such that
		\begin{equation} \label{eq:mixedterms}
			\sum_k \abs{c_{n,k}}^2 \leq \frac{\inf_k\abs{g_{n,k}(1)}^2}{\epsilon^2}\Bigl\lvert \sum_{k} c_{n,k} \Bigr\rvert^2.
		\end{equation}
		Then there is Lipschitz stability in $B$,
		\begin{equation*}
			\norm{\gamma_1-\gamma_2}_{L^2(B)} \leq \epsilon^{-1}\abs{\partial B}^{1/2}C_{\partial B}\norm{\Lambda(\gamma_1)-\Lambda(\gamma_2)}_{\mathscr{L}(H^{1/2}(\partial B),H^{-1/2}(\partial B))},
		\end{equation*}
		with $\abs{\partial B}$ being the surface measure of $\partial B$ and $C_{\partial B}$ the constant from Theorem \ref{thm:bdrstab}.
	\end{theorem}
	\begin{proof}
		The proof is given in Section~\ref{sec:proofmain}.
	\end{proof}
	The condition \eqref{eq:mixedterms} is equivalent to
	\begin{equation*}
		\sum_{k,k' \text{ with } k\neq k'} \redel(c_{n,k}\overline{c_{n,k'}}) \geq \Bigl(\frac{\epsilon^2}{\inf_k\abs{g_{n,k}(1)}^2}-1\Bigr)\sum_{k}\abs{c_{n,k}}^2,
	\end{equation*}
	and is a peculiar condition: for any fixed $n$, not too much cancellation is allowed for the mixed $k$-indices, indicating a limited freedom in the behavior of $\gamma_1-\gamma_2$ in the radial direction. As an example, if the coefficients $c_{n,k}$ are real and their signs only depend on $n$ (ignoring vanishing coefficients), then the condition is guaranteed to hold with 
	\begin{equation*}
		\epsilon = \inf_{n,k}\abs{g_{n,k}(1)}.
	\end{equation*}
	Moreover, in $d=2$ spatial dimensions such examples of Lipschitz stability can be transferred to a general bounded simply-connected $C^{1,\alpha}$ domain by using the Kellogg--Warschawski theorem (see,~e.g.,~\cite{GH2024} for details in a linearized setting).
		
	Condition \eqref{eq:mixedterms} simplifies when nonzero coefficients are only allowed for a fixed $k$. Let
	\begin{equation*}
		W_k = \overline{\spanm\{\psi_{n,k}\}_n}
	\end{equation*}
	and define
	\begin{equation*}
		\mathcal{A}_k = \bigl\{\, f \in W_k \mid (c_n) \in \ell^1_{a_{n,k}} \text{ where } c_n = \inner{f,\psi_{n,k}}_{L^2(B)} \,\bigr\}.
	\end{equation*}
	If $(\gamma_1-\gamma_2) \in \mathcal{A}_k$ then \eqref{eq:mixedterms} is replaced with
	\begin{equation*}
		\abs{g_{n,k}(1)} \geq \epsilon,
	\end{equation*}
	which holds with $\epsilon = \inf_{n}\abs{g_{n,k}(1)}$. This is summarized in the following corollary that provides Lipschitz stability for conductivity changes in infinite-dimensional subspaces of $L^2(B)$.
	
	\begin{corollary} \label{cor2}
		Let $\gamma_1,\gamma_2\in C_+(\overline{B})$. If $(\gamma_1-\gamma_2)\in \mathcal{A}_k$ there is Lipschitz stability in $B$,
		\begin{equation*}
			\norm{\gamma_1-\gamma_2}_{L^2(B)} \leq \frac{\abs{\partial B}^{1/2}}{\inf_{n}\abs{g_{n,k}(1)}}C_{\partial B}\norm{\Lambda(\gamma_1)-\Lambda(\gamma_2)}_{\mathscr{L}(H^{1/2}(\partial B),H^{-1/2}(\partial B))}.
		\end{equation*}
	\end{corollary}
	
	In particular, if $(\gamma_1-\gamma_2)\in \cup_k \mathcal{A}_k$, one can use the common stability constant
	\begin{equation*}
		\frac{\abs{\partial B}^{1/2}}{\inf_{n,k}\abs{g_{n,k}(1)}}C_{\partial B}.
	\end{equation*}
	Once again, we could have
	\begin{equation*}
		\gamma_1 = \gamma_0 + \widetilde{\gamma}_1 \quad \text{and} \quad \gamma_2 = \gamma_0 + \widetilde{\gamma}_2,
	\end{equation*}
	where, e.g., $\widetilde{\gamma}_1,\widetilde{\gamma}_2\in \mathcal{A}_k\cap C(\overline{B})$ for some $k$, and $\gamma_0$ is a large enough $C_+(\overline{B})$ function to ensure that the conductivities are bounded away from zero.
	
	Concluding that Theorem~\ref{thm:main} and Corollary~\ref{cor2} are actually applicable requires concrete examples where Assumption~\ref{assump} is satisfied. This is tackled in Section~\ref{sec:zernikemain} where we first derive an orthonormal Zernike basis for $L^2(B)$ in any spatial dimension $d\geq 2$, and subsequently verify that such a basis satisfies Assumption~\ref{assump}. Sections~\ref{sec:zernike2d}~and~\ref{sec:zernike3d} are dedicated to the Zernike basis functions in two and three spatial dimensions, which are the most relevant settings for the Calder\'on problem. We moreover give several fundamental properties of the Zernike basis in Theorem~\ref{thm:zernikebnds}.
	
	\subsection{Earlier stability results} \label{sec:background}
	
	For very general infinite-dimensional classes of conductivities the Calder\'on problem exhibits conditional logarithmic stability \cite{Alessandrini1988}, and this is optimal for such general settings \cite{Mandache2001,Koch2021}. For measurements on a subset of the domain boundary, only double-logarithmic stability has been proved \cite{HeckWang2006}, unless the conductivities are known near the domain boundary, in which case it is possible to again obtain logarithmic stability \cite{Alessandrini2012}. However, we emphasize that for smaller classes of conductivities the stability can be improved and there is no restriction why such classes cannot still be infinite-dimensional, as observed in this paper. To our knowledge, Lipschitz stability has not previously been documented for infinite-dimensional classes of conductivities for the nonlinear Calder\'on problem. Note that Lipschitz stability, instead of, e.g., logarithmic stability, is also important for convergence guarantees in reconstruction using iterative methods (see, e.g.,~\cite{Hoop2012}).
	
	In a linearized problem for the spatial dimension $d=2$, we previously proved that there are infinite-dimensional function spaces of perturbations (also expanded in a Zernike basis) for which there is Lipschitz stability \cite{GH2024}. However, this stability was with respect to a Hilbert--Schmidt norm for the linearized forward map (based on a Neumann-to-Dirichlet map), which prevented generalizations to $d>2$. It was at that time unclear if such results could be carried over to the nonlinear problem, in particular for the operator norm and any spatial dimension $d\geq 2$. On the other hand, the stability results for the linearized problem do not require the additional decay condition for the basis coefficients.
	
	For finite-dimensional classes of conductivities there have been many results on Lipschitz stability. A non-exhaustive list of such papers include \cite{Harrach_2019, Harrach2023,Alberti2019,Alberti2020,Alberti2022,Alberti2023,Alessandrini2005,Alessandrini2017,Beretta2011}. We emphasize that in many such cases it is also possible to reduce the number of measurements from a Dirichlet-to-Neumann or Neumann-to-Dirichlet map to just a finite set of boundary measurements.
	
	For the related easier problem of inclusion detection, where the support of a perturbation is reconstructed rather than the conductivity values, Lipschitz stability has been proved in quite general settings of polygonal, polyhedral, and small volume fraction inclusions \cite{Hanke2024,Beretta2022,Beretta2020,Beretta2021,Friedman1989b}. We emphasize, in particular, the recent paper \cite{Hanke2024} for Lipschitz stability in the determination of a polygonal inclusion from just two boundary measurements.

    \subsection{Article structure}
    To summarize, the rest of this text is organized as follows. Section~\ref{sec:zernikemain} constructs Zernike bases that satisfy Assumption~\ref{assump} for all $d \geq 2$. Section~\ref{sec:proofmain} provides a proof for Theorem~\ref{thm:main}. Section~\ref{sec:proofthmzernike} provides a proof for Theorem~\ref{thm:zernikebnds} that considers properties of the constructed radial Zernike polynomials. The proof of Proposition 2.1 on an $L^\infty(\partial B)$-bound for spherical harmonics is presented in Appendix~\ref{sec:appA}.
	
	\section{Zernike basis} \label{sec:zernikemain}
	
	\subsection{Spherical harmonics} \label{sec:sphericalham}
	
	As a preparation to the introduction of the Zernike basis, we recall some facts about spherical harmonics; see \cite[Chapter~5]{Axler_2001} and \cite{Efthimiou_2014} for additional insights.
	
	A homogeneous polynomial $p \colon \R^d\to\C$ of degree $\ell$ satisfies $p(\lambda x) = \lambda^\ell p(x)$ for any $\lambda\in\R$, or equivalently
	\begin{equation*}
		p(x) = \sum_{|\alpha| = \ell} c_\alpha x^\alpha,
	\end{equation*}
	for coefficients $c_\alpha\in \C$ with $\alpha\in \N_0^d$. We use multi-index notation, i.e.\ denote $x^\alpha = \prod_{i=1}^d x_i^{\alpha_i}$ and $\abs{\alpha} = \sum_{i=1}^d \alpha_i$ for $\alpha\in \N_0^d$. Let
	\begin{equation*} 
		\mathcal{P}_\ell = \Bigl\{\, p \colon \R^d \to \C \mid p(x) = \sum_{|\alpha| = \ell} c_\alpha x^\alpha, \; c_\alpha \in \C, \; \Delta p = 0 \,\Bigr\}
	\end{equation*}
	be the space of \emph{harmonic} homogeneous polynomials of degree $\ell$. 
	
	The space of spherical harmonics of degree $\ell$ is defined as the restriction of the harmonic homogeneous polynomials of degree $\ell$ to the unit sphere:
	\begin{equation*} 
		\mathcal{H}_\ell = \{\, p|_{\partial B} \mid p \in \mathcal{P}_\ell \,\}.
	\end{equation*}
	We have 
	\begin{equation} \label{dimHl}
		\dim\mathcal{H}_\ell = \dim\mathcal{P}_\ell = \binom{\ell+d-1}{d-1} - \binom{\ell+d-3}{d-1},
	\end{equation}
	with the convention that $\binom{m}{k} = 0$ for nonnegative integers $m<k$.
	
	The $\mathcal{H}_\ell$-spaces are the eigenspaces for the Laplace–Beltrami operator $\Delta_{\partial B}$, with
	\begin{equation} \label{eq:LBeig}
		\Delta_{\partial B}f = -\ell(\ell+d-2)f, \quad f\in\mathcal{H}_\ell.
	\end{equation}
	Hence, there is the following orthogonal decomposition:
	\begin{equation*}
		L^2(\partial B) = \bigoplus_{\ell=0}^\infty \mathcal{H}_\ell.
	\end{equation*}
	Let $\{f_{\ell,m}\}_{m\in\mathcal{I}_{\ell}}$ be an orthonormal basis for $\mathcal{H}_\ell$; in particular, the index set $\mathcal{I}_{\ell}$ has the cardinality $\dim\mathcal{H}_\ell$. Thereby, $\{f_{\ell,m}\}_{\ell\in\N_0, m\in\mathcal{I}_{\ell}}$ is an orthonormal basis for $L^2(\partial B)$.
	
	We will need the following bound in $L^\infty$.
	
	\begin{proposition} \label{prop:sphericalhambnds}
		For $f\in\mathcal{H}_\ell$, there is the sharp bound
		\begin{equation*}
			\norm{f}_{L^\infty(\partial B)} \leq \Bigl(\frac{\dim\mathcal{H}_\ell}{\abs{\partial B}}\Bigr)^{1/2}\norm{f}_{L^2(\partial B)}.
		\end{equation*}
	\end{proposition}
	\begin{proof}
		The proof is given in Appendix~\ref{sec:appA}.
	\end{proof}
	
	\subsection{Constructing the Zernike basis} \label{sec:zernike}
	
	Consider the $d$-dimensional Zernike operator 
	\begin{equation*}
		Z = -\Delta + (x \cdot \nabla)^2 + d(x \cdot \nabla)
	\end{equation*}
	in the unit ball domain $B\subset \R^d$, generalized from the two-dimensional case as originally considered in \cite[Section~6]{Zernike1934}. Here $x \cdot\nabla$ denotes the differential operator
	\begin{equation*}
		x\cdot\nabla = \sum_{i=1}^d x_i\frac{\partial}{\partial x_i}.
	\end{equation*}
	The $d$-dimensional Zernike basis for $L^2(B)$ is comprised of normalized eigenfunctions for $Z$, i.e.\ we consider the eigenvalue problem
	\begin{equation*}
		Z\psi = \mu\psi.
	\end{equation*}
	Despite the fact that Zernike bases can be introduced for any $d\in\N\setminus\{1\}$, they are mostly unexplored for $d>2$. For $d=3$ see, e.g., \cite{Garde2025,Mathar2009}. We construct the Zernike basis for a general spatial dimension $d\in\N\setminus\{1\}$.
	
	In polar coordinates the eigenvalue problem takes the form
	\begin{equation} 
		\label{eq:polar_Z}
		\Bigl( \frac{\partial^2}{\partial r^2} + \frac{d - 1}{r} \frac{\partial}{\partial r} + \frac{1}{r^2} \Delta_{\partial B} \Bigr) \psi - \Bigl( r \frac{\partial}{\partial r} \Bigr)^2 \psi - dr \frac{\partial}{\partial r} \psi = -\mu\psi.
	\end{equation}
	Using separation of variables with $\psi(r,\theta) = R(r)\Theta(\theta)$, we may expand $\Theta$ in the orthonormal spherical harmonics basis $\{f_{\ell,m}\}_{\ell\in\N_0,m\in\mathcal{I}_\ell}$ for $L^2(\partial B)$. Since $f_{\ell,m}$ is an eigenfunction for $\Delta_{\partial B}$ with the eigenvalue 
	\begin{equation*}
		\lambda_\ell = -\ell(\ell + d - 2),
	\end{equation*}
	this approach reduces \eqref{eq:polar_Z} to an $\ell$-dependent radial problem on the interval $(0,1$). After multiplying with $r^{d-1}$, one ends up with the equation
	\begin{equation} 
		r^{d-1}(1-r^2) \frac{\di^2}{\di r^2} R_\ell + \bigl[(d - 1) r^{d-2} - (d + 1) r^d\bigr] \frac{\di}{\di r}R_\ell + \lambda_\ell r^{d-3} R_\ell = -\mu r^{d-1} R_\ell. \label{eq:poly_radial}
	\end{equation}
	Rearranging the terms leads to a singular (at $r=0$ and $r=1$) Sturm--Liouville eigenvalue problem
	\begin{equation} \label{eq:SL-form}
		\frac{\di}{\di r} \Bigl[ r^{d-1} (1 - r^2) \frac{\di}{\di r} R_\ell \Bigr] + \lambda_\ell r^{d-3} R_\ell = -\mu r^{d-1} R_\ell,
	\end{equation}
	which is self-adjoint, with eigenvalue $\mu$ and weight-function $r^{d-1}$. By virtue of Sturm--Liouville theory, this entails that for each $\ell\in\N_0$, there are discrete real eigenvalues $\{ \mu_{\ell,k} \}_{k\in\N_0}$, and one can construct an orthonormal basis for $L^2_{r^{d-1}}([0,1])$ from the corresponding eigenfunctions for \eqref{eq:SL-form}. 
	
	We will first determine certain bounded orthonormal eigenfunctions $\{R_{\ell,k}\}_{k\in\N_0}$ for \eqref{eq:SL-form}, without worrying about whether we find all the orthogonal eigenfunctions. Subsequently, we will prove that the constructed $\{R_{\ell,k}\}_{k\in\N_0}$ indeed forms an orthonormal basis for $L^2_{r^{d-1}}([0,1])$. As $\{f_{\ell,m}\}_{\ell\in\N_0,m\in\mathcal{I}_\ell}$ is an orthonormal basis for $L^2(\partial B)$, then 
	\begin{equation*}
		\psi_{\ell,m,k}(r\theta) = R_{\ell,k}(r)f_{\ell,m}(\theta), \quad \ell,k\in\N_0,\enskip m\in\mathcal{I}_\ell
	\end{equation*}
	becomes an orthonormal basis for $L^2(B)$.
	
	We may solve \eqref{eq:poly_radial}, and thus \eqref{eq:SL-form}, with the following strategy. Consider the hypergeometric differential equation, with parameters $a$, $b$, and $c$, on the interval $(0,1)$,
	\begin{equation} \label{eq:hypergeom}
		z(1 - z)\frac{\di^2}{\di z^2}w + [c - (a + b + 1)z]\frac{\di}{\di z}w - abw = 0,
	\end{equation}
	which admits a solution given by the Gauss hypergeometric function 
	\begin{equation*}
		w(z) = \F(a,b;c;z),
	\end{equation*}
	assuming $c$ is not a nonpositive integer. It is known that linear second order ordinary differential equations with three regular singularities, e.g., at $0$, $1$, and $\infty$, can be transformed to a hypergeometric differential equation. By making the change of variable $z = r^2$, and $h(r) = w(r^2)$, we transform \eqref{eq:hypergeom} into 
	\begin{equation} \label{eq:hypergeom_r2}
		(1 - r^2)\frac{\di^2}{\di r^2}h + \Bigl[ \frac{2c - 1}{r} - 2 \Bigl(a + b + \frac{1}{2} \Bigr) r \Bigr] \frac{\di}{\di r}h - 4abh = 0.
	\end{equation}
	By inserting
	\begin{equation} \label{eq:w_translation}
		R(r) = r^{\ell}h(r) = r^{\ell}w(r^2)
	\end{equation}
	into \eqref{eq:hypergeom_r2} and multiplying the equation with $r^{\ell+d-1}$, we arrive at
	\begin{align} 
		r^{d-1}(1-r^2)\frac{\di^2}{\di r^2}R + [ (2c - 2\ell - 1)r^{d-2} &- (2a + 2b - 2\ell + 1)r^d]\frac{\di}{\di r}R \notag\\
		&- \ell(2c - \ell - 2)r^{d-3}R = -(\ell - 2a)(2b - \ell)r^{d-1}R. \label{eq:hypergeom_r2_tra}
	\end{align}
	For the coefficients of \eqref{eq:poly_radial} and \eqref{eq:hypergeom_r2_tra} to match, we must have
	\begin{align*}
		2c-2\ell &= d, \\
		2a+2b-2\ell &= d, \\
		\ell(2c-\ell) &= \ell(\ell + d), \\
		(\ell-2a)(2b-\ell) &= \mu,
	\end{align*}
	where the first condition makes the third one redundant. In consequence,
	\begin{equation*}
		c = a+b = \ell + \tfrac{d}{2}
	\end{equation*}
	and
	\begin{equation*}
		\mu = -4ab+2(a+b)\ell-\ell^2 = -4ab+\ell^2+d\ell.
	\end{equation*}
	Since $c=a+b$, it follows from \cite[§15.4(ii) Eq.~15.4.21 and Eq.~15.4.24]{NIST} that we must insist that
	\begin{equation*}
		a = -k, \quad k\in\N_0
	\end{equation*}
	as we aim for bounded eigenfunctions and thus want to avoid a logarithmic singularity in the hypergeometric function at $z=1$. Thus, we have
	\begin{equation*}
		b = c-a = \ell+k+\tfrac{d}{2},
	\end{equation*} 
	and the eigenvalues take the form
	\begin{equation*}
		\mu_{\ell,k} = -4(-k)(\ell+k+\tfrac{d}{2})+\ell^2+d\ell = (\ell + 2k)(\ell + 2k + d).
	\end{equation*}
	Altogether, with a normalization constant $C_{\ell,k}$, we have constructed the orthonormal eigenfunctions
	\begin{equation} \label{eq:Zernike_sol}
		R_{\ell,k}(r) = C_{\ell,k}r^\ell \F(-k, \ell + k + \tfrac{d}{2}; \ell + \tfrac{d}{2}; r^2), \quad k \in \N_0
	\end{equation}
	in $L^2_{r^{d-1}}([0,1])$ for the Sturm--Liouville problem \eqref{eq:SL-form}, as the eigenfunctions correspond to different eigenvalues. It follows from \cite[§15.2(ii) Eq.~15.2.4]{NIST} that $R_{\ell,k}$ is a polynomial of degree $\ell+2k$, which we call a \emph{radial Zernike polynomial},
	\begin{equation} \label{eq:Zernike1}
		R_{\ell,k}(r) = C_{\ell,k}\sum_{q=0}^k (-1)^q\binom{k}{q}\frac{(\ell+k+\tfrac{d}{2})_q}{(\ell+\tfrac{d}{2})_q}r^{\ell+2q},
	\end{equation}
	written in terms of Pochhammer symbols (rising factorials).
	
	We have collected a number of fundamental properties of the radial Zernike polynomials below. 
	
	\begin{theorem} \label{thm:zernikebnds} {}\
		\begin{enumerate}[\rm(i)]
			\item For each $\ell\in\N_0$, $\{R_{\ell,k}\}_{k\in\N_0}$ is an orthonormal basis for $L^2_{r^{d-1}}([0,1])$.
			\item The normalizing constant is
			\begin{equation*}
				C_{\ell,k} = (-1)^k\sqrt{2\ell+4k+d}\binom{\ell+k+\tfrac{d-2}{2}}{k},
			\end{equation*}
			where $(-1)^k$ ensures that $R_{\ell,k}(1)$ is positive, with the function value
			\begin{equation*}
				R_{\ell,k}(1) = \sqrt{2\ell+4k+d}.
			\end{equation*}
			This choice of sign is assumed for the properties {\rm (iii)} and {\rm (v)} below.
			\item We have
			\begin{equation*}
				\frac{\di}{\di r}R_{\ell,k}(1) > 0,
			\end{equation*}
			except for $\ell=k=0$, in which case $R_{0,0}$ is a constant function.
			\item There is the uniform bound
			\begin{equation*}
				\max_{r\in[0,1]}\abs{R_{\ell,k}(r)} \leq \abs{C_{\ell,k}}. 
			\end{equation*}
			\item For any $\ell, p \in \N_0$,
			\begin{equation*}
				r^{\ell + 2p} = \sum^p_{k = 0} \chi_{\ell, k, p} \, R_{\ell, k}(r),
			\end{equation*} 
			where
			\begin{equation*}
				\chi_{\ell, k, p} = \frac{\sqrt{2\ell + 4k + d}}{2p + 2\ell + 2k + d} \binom{p}{k} \binom{p + \ell + k + \frac{d-2}{2}}{k}^{-1}.
			\end{equation*}
		\end{enumerate}
	\end{theorem}
	\begin{proof}
		The proof is given in Section~\ref{sec:proofthmzernike}.
	\end{proof}
	In particular, for $\ell,k\in\N_0$, the radial Zernike polynomials are explicitly given as
	\begin{align}
		R_{\ell,k}(r) &= \sqrt{2\ell+4k+d}\sum_{q=0}^k (-1)^{k+q}\binom{k}{q}\binom{\ell+k+\tfrac{d-2}{2}}{k}\frac{(\ell+k+\tfrac{d}{2})_q}{(\ell+\tfrac{d}{2})_q}r^{\ell+2q} \notag\\
		&= \sqrt{2\ell+4k+d}\sum_{s=0}^k (-1)^s\binom{k}{s}\binom{\ell+2k-s+\tfrac{d-2}{2}}{k}r^{\ell+2k-2s}, \label{eq:zernikepol}
	\end{align}
	where the summation index was changed to $s = k-q$, and the Pochhammer symbols and binomial coefficients were written out and rearranged in the latter representation.
	
	For $\ell,k\in\N_0$ and $m\in\mathcal{I}_\ell$, the orthonormal Zernike basis functions for $L^2(B)$ are
	\begin{align*}
		\psi_{\ell,m,k}(r\theta) &= R_{\ell,k}(r)f_{\ell,m}(\theta).
	\end{align*}
	
	\subsection{Connecting the Zernike basis to Assumption~\ref{assump}} \label{sec:ZerniketoStab}
	
	To show how the Zernike basis fits into the stability results in Section~\ref{sec:results}, we let
	\begin{equation*}
		n = (\ell,m)
	\end{equation*}
	for $\ell\in\N_0$ and $m\in\mathcal{I}_\ell$, and consider $k\in\N_0$. 
	In the notation of Assumption~\ref{assump}, choosing
	\begin{equation*}
		f_n = f_{\ell,m} \quad \text{and} \quad g_{n,k} = R_{\ell,k}
	\end{equation*}
	gives the Zernike basis functions as $\psi_{n,k}$. In particular,
	\begin{equation*}
		W = L^2(B).
	\end{equation*}
	According to Proposition~\ref{prop:sphericalhambnds} and Theorem~\ref{thm:zernikebnds},
	\begin{equation} \label{eq:psiZernikebnd}
		\norm{\psi_{n,k}}_{L^\infty(B)} \leq \abs{C_{\ell,k}}\Bigl(\frac{\dim\mathcal{H}_\ell}{\abs{\partial B}}\Bigr)^{1/2},
	\end{equation}
	which we may directly use as our $a_{n,k}$. However, this is actually nonoptimal since Theorem~\ref{thm:zernikebnds} also shows that in some interior neighborhood of $\partial B$ the maximum of $\abs{R_{\ell,k}}$ is attained at $r=1$ (although such a neighborhood may depend on both $\ell$ and $k$). On a fixed spherical shell $U$ as in Assumption~\ref{assump}, one can therefore replace $|C_{\ell,k}|$ in $a_{n,k}$ with $R_{\ell,k}(1)=\sqrt{2\ell+4k+d}$ for many of the indices.
	
	From Theorem~\ref{thm:zernikebnds} we have
	\begin{equation*}
		\inf_{n,k}\abs{g_{n,k}(1)} = \sqrt{d},
	\end{equation*}
	which completes checking that the Zernike basis satisfies all conditions in Assumption~\ref{assump}. In particular, the stability constant in Corollary~\ref{cor2} takes the form
	\begin{equation}
		\Bigl(\frac{\abs{\partial B}}{4k+d}\Bigr)^{1/2}C_{\partial B}.	
	\end{equation}
	Even though stability actually improves as $k$ increases, the decay condition for the basis coefficients in $\mathcal{A}_k$ also becomes stronger.
	
	\subsection{The two-dimensional Zernike basis} \label{sec:zernike2d}
	
	The case $d=2$ for the Zernike basis is quite well-known in the literature, e.g., for the Calder\'on problem \cite{GH2024,Autio2025,Bisch,Allers91}, Computed Tomography \cite{Cormack1964,Louis1981}, and especially for problems in optics \cite{Zernike1934,Bhatia1954,BornWolf}. 
	
	According to \eqref{dimHl} we have $\dim\mathcal{H}_0 = 1$ and $\dim\mathcal{H}_\ell = 2$ for $\ell\in\N$, which leads naturally to the choices
	\begin{equation*}
		\ell = \abs{j}, \quad \mathcal{I}_\ell = \{-j,j\}, \quad \mathcal{I}_0=\{0\}.
	\end{equation*}
	Thereby, $\ell$ and $m$ are combined into a single index $j\in\Z$. Moreover, we identify $\theta\in\partial B$ with $\theta\in[0,2\pi)$, as is typical for polar coordinates in $\R^2$. Hence, we may choose the standard orthonormal basis for the spherical harmonics,
	\begin{equation*}
		f_{\ell,m}(\theta) = \frac{\e^{\I j\theta}}{\sqrt{2\pi}}.
	\end{equation*}
	The two-dimensional Zernike basis functions, with $j\in\Z$ and $k\in\N_0$, thus become:
	\begin{align} \label{eq:Zernikebasis2D}
		\psi_{j,k}^{\textup{2D}}(r,\theta) &= \sqrt{2\abs{j}+4k+2}\sum_{s=0}^k (-1)^s\binom{k}{s}\binom{\abs{j}+2k-s}{k}r^{\abs{j}+2k-2s}\frac{\e^{\I j\theta}}{\sqrt{2\pi}} \notag \\
		&= \sqrt{2\abs{j}+4k+2}\sum_{s=0}^k (-1)^s\binom{\abs{j}+2k-s}{s}\binom{\abs{j}+2k-2s}{k-s}r^{\abs{j}+2k-2s}\frac{\e^{\I j\theta}}{\sqrt{2\pi}}, 
	\end{align}
	where the latter expression is perhaps the more common one for $d=2$ (cf.~\cite{GH2024}).
	
	\subsection{The three-dimensional Zernike basis} \label{sec:zernike3d}
	
	In the case $d=3$, it is customary to identify a point on $\partial B$ with its spherical coordinates, i.e.\ the polar angle $\theta\in[0,\pi]$ and the azimuthal angle $\varphi\in[0,2\pi)$.
	
	From \eqref{dimHl} we have $\dim\mathcal{H}_\ell = 2\ell+1$, and we may thus choose $\mathcal{I}_\ell = \{-\ell,\dots,\ell\}$. This fits well with the standard orthonormal basis $\{ Y^m_\ell \}_{m \in \mathcal{I}_\ell}$ for $\mathcal{H}_\ell$, consisting of Laplace's spherical harmonics of degree $\ell$ and order $m$. 
	
	The three-dimensional Zernike basis functions, for $\ell,k\in\N_0$ and $m\in\{-\ell,\dots,\ell\}$, therefore become (cf.~\cite{Garde2025,Mathar2009}):
	\begin{equation} \label{eq:Zernikebasis3D}
		\psi_{\ell,m,k}^{\textup{3D}}(r,\theta,\varphi) = \sqrt{2\ell+4k+3}\sum_{s=0}^k (-1)^s\binom{k}{s}\binom{\ell+2k-s+\tfrac{1}{2}}{k}r^{\ell+2k-2s}Y^m_\ell(\theta,\varphi).
	\end{equation}
	
	\section{Proof of Theorem~\ref{thm:main}} \label{sec:proofmain}
	
	The decay condition in $\mathcal{A}$ allows for a boundary trace to be taken inside a basis expansion.
	\begin{lemma} \label{lemma:trace}
		Let $f\in \mathcal{A}\cap C(\overline{B})$ and $c_{n,k}=\inner{f,\psi_{n,k}}_{L^2(B)}$. Then,
		\begin{equation*}
			f(\theta) = \sum_{n,k} c_{n,k}\psi_{n,k}(\theta) = \sum_{n,k} c_{n,k}g_{n,k}(1)f_n(\theta), \quad \text{a.e. } \theta\in\partial B.
		\end{equation*}
	\end{lemma}
	\begin{proof}
		By the assumed continuity of $f$, we have $f(\theta) = \lim_{r\to 1}f(r\theta)$. 
		The condition $(c_{n,k}) \in \ell_{a_{n,k}}^1$, together with Assumptions~\ref{assump}(iii) and \ref{assump}(v), allows one to apply the Lebesgue dominated convergence theorem and take the radial limit inside the expansion of $f$ in the basis $\{\psi_{n,k}\}_{n,k}$.
	\end{proof}
	
	We now proceed to prove Theorem~\ref{thm:main}. According to Lemma~\ref{lemma:trace},
	\begin{equation}
		\label{eq:g_expand}
		(\gamma_1-\gamma_2)|_{\partial B} = \sum_{n} \Bigl(\sum_k c_{n,k}g_{n,k}(1)\Bigr)f_n.
	\end{equation}
	$\{f_n\}_n$ and $\{\psi_{n,k}\}_{n,k}$ being orthonormal bases for $L^2(\partial B)$ and $W$, respectively, implies
	\begin{equation*}
		\norm{\gamma_1-\gamma_2}_{L^2(\partial B)}^2 = \sum_{n} \Bigl|\sum_k c_{n,k}g_{n,k}(1)\Bigr|^2 
	\end{equation*}
	and
	\begin{equation*}
		\norm{\gamma_1-\gamma_2}_{L^2(B)}^2 = \sum_{n,k} \abs{c_{n,k}}^2.
	\end{equation*}
	By combining these equalities with condition~\eqref{eq:mixedterms}, it follows that 
	\begin{align*}
		\norm{\gamma_1-\gamma_2}_{L^2(B)}^2 &  
		\leq \epsilon^{-2}\sum_n \inf_k\abs{g_{n,k}(1)}^2\Bigl\lvert\sum_k c_{n,k}\Bigr\rvert^2 \\
		&= \epsilon^{-2}\sum_n \inf_k\abs{g_{n,k}(1)}^2\Bigl\lvert\sum_k g_{n,k}(1)^{-1}g_{n,k}(1)c_{n,k}\Bigr\rvert^2 \\
		&\leq \epsilon^{-2}\sum_{n} \Bigl|\sum_k c_{n,k}g_{n,k}(1)\Bigr|^2 \\
		&= \epsilon^{-2}\norm{\gamma_1-\gamma_2}_{L^2(\partial B)}^2 \\[1mm]
		&\leq \epsilon^{-2}\abs{\partial B}\norm{\gamma_1-\gamma_2}_{L^\infty(\partial B)}^2.
	\end{align*}
	The proof is concluded via Theorem~\ref{thm:bdrstab}.
	
	\section{Proof of Theorem~\ref{thm:zernikebnds}} \label{sec:proofthmzernike}
	
	In this proof, let $\inner{\,\cdot\,,\,\cdot\,}$ be the $L^2_{r^{d-1}}([0,1])$ inner product and let $\norm{\,\cdot\,}$ be the associated norm. When we refer to ``orthonormality'' it is to be understood in the sense of $L^2_{r^{d-1}}([0,1])$.
	
	\subsection*{(i)} 
	
	We have already established that $\{R_{\ell,k}\}_{k\in\N_0}$ is an orthonormal set as it is composed of eigenfunctions for the Sturm--Liouville eigenvalue problem \eqref{eq:SL-form}. What remains to be proven is density in $L^2_{r^{d-1}}([0,1])$.  
	
	Due to \eqref{eq:Zernike1}, 
	\begin{equation*}
		\spanm\{R_{\ell,k}\}_{k=0}^p \subseteq \spanm\{r^{\ell+2k}\}_{k=0}^p, \quad p\in\N_0,
	\end{equation*}
	where the latter vector space clearly is $(p+1)$-dimensional. Since $\{R_{\ell,k}\}_{k=0}^p$ are $p+1$ orthonormal functions, they actually form a basis for $\spanm\{r^{\ell+2k}\}_{k=0}^p$. As the weight $r^{d-1}$ is bounded on $(0,1)$, the density of $\{R_{\ell,k}\}_{k\in \N_0}$ in $L^2_{r^{d-1}}([0,1])$ is a consequence of the density of 
	\begin{equation*}
		\spanm\{\, r^{\ell+2k} \mid k\in\N_0 \,\}
	\end{equation*}
	in $L^2([0,1])$. The latter was proved in \cite[Lemma~6.1]{GH2024}.
	
	\subsection*{(ii)}  
	
	For a fixed $\ell\in\N_0$ we define
	\begin{align*}
		F_k(r) &= \F(-k,\ell+k+\tfrac{d}{2};\ell+\tfrac{d}{2};r^2), \\
		\widetilde{F}_k(r) &= \F(-(k-1),\ell+k+\tfrac{d}{2};\ell+\tfrac{d}{2};r^2).
	\end{align*}
	Note that $R_{\ell,k} = C_{\ell,k}r^\ell F_k$. Because the case $k=0$ follows directly from $F_0\equiv 1$, we may assume that $k\in\N$ in the following. 
	
	Consider the identity \cite[§15.5(ii) Eq.~15.5.12]{NIST} (with $a+1$ instead of $a$ and $b-1$ instead of $b$):
	\begin{equation*}
		(b-1)\F(a+1,b;c;z) = (b-a-2)\F(a+1,b-1;c;z) + (a+1)\F(a+2,b-1;c;z),
	\end{equation*}
	which in the context of this proof gives
	\begin{equation} \label{eq:Fjident}
		\widetilde{F}_j = \frac{\ell+2j+\tfrac{d-2}{2}-1}{\ell+j+\tfrac{d-2}{2}}F_{j-1} - \frac{j-1}{\ell+j+\tfrac{d-2}{2}}\widetilde{F}_{j-1}.
	\end{equation}
	Since $r^{\ell}F_{j-1}$ and $r^{\ell}F_k$ are orthogonal for $j=1,\dots,k$ by virtue of the the Sturm--Liouville eigenvalue problem \eqref{eq:SL-form}, we have
	\begin{equation} \label{eq:Fjident2}
		\inner{r^{\ell}\widetilde{F}_j,r^{\ell}F_k} = -\frac{j-1}{\ell+j+\tfrac{d-2}{2}}\inner{r^{\ell}\widetilde{F}_{j-1},r^{\ell}F_k}, \quad j=1,\dots,k.
	\end{equation}
	As \eqref{eq:Fjident2} vanishes for $j=1$, by induction it follows that
	\begin{equation} \label{eq:Fjip}
		\inner{r^{\ell}\widetilde{F}_j,r^{\ell}F_k} = 0,\quad j=1,\dots,k.
	\end{equation}
	
	Next, we need the following identity that is based on \cite[§15.5(ii) Eq.~15.5.16\_5]{NIST} (with $a+1$ instead of $a$) and \cite[§15.5(i) Eq.~15.5.1]{NIST}:
	\begin{equation*}
		z\frac{\di}{\di z}\F(a,b;c;z) = a\bigl(\F(a+1,b;c;z)-\F(a,b;c;z)\bigr).
	\end{equation*}
	Setting $z=r^2$ leads to
	\begin{equation} \label{eq:diffFk}
		r\frac{\di}{\di r}F_k = -2k\bigl(\widetilde{F}_k-F_k\bigr)
	\end{equation}
	in the context of this proof.
	Using integration by parts and the formulas \eqref{eq:Fjip} and \eqref{eq:diffFk}, 
	\begin{align*}
		d\norm{r^{\ell}F_k}^2 &= d\int_0^1 (r^\ell F_k)^2 r^{d-1}\,\di r \\
		&= F_k(1)^2 - 2\int_0^1 (r^\ell F_k)\Bigl(\ell r^{\ell-1}F_k+r^\ell \frac{\di}{\di r}F_k \Bigr)r^d\,\di r \\
		&= F_k(1)^2 - (2\ell+4k)\norm{r^\ell F_k}^2,
	\end{align*}
	meaning that
	\begin{equation} \label{eq:rFnorm}
		\norm{r^\ell F_k}^2 = \frac{F_k(1)^2}{2\ell+4k+d}.
	\end{equation}
	By the Chu--Vandermonde identity \cite[§15.4(ii) Eq.~15.4.24]{NIST},
	\begin{equation*}
		F_k(1) = \frac{(-k)_k}{(\ell+\tfrac{d}{2})_k} = (-1)^k\binom{\ell+k+\tfrac{d-2}{2}}{k}^{-1},
	\end{equation*}
	which together with \eqref{eq:rFnorm} proves the claimed formulas for $C_{\ell,k}$ and $R_{\ell,k}(1)$.
	
	\subsection*{(iii)}  
	
	First of all, it is clear from \eqref{eq:Zernike1} that $R_{0,0}$ is a constant function. 
	
	Consider the derivative
	\begin{equation} \label{eq:Rderiv}
		\frac{\di}{\di r}R_{\ell,k} = C_{\ell,k}\ell r^{\ell-1}F_k + C_{\ell,k}r^{\ell}\frac{\di}{\di r}F_k. 
	\end{equation}
	As we have already proved that $C_{\ell,k}F_k(1)>0$, the first term on the right-hand side vanishes for $\ell = 0$ and is otherwise positive at $r=1$.
	Since $F_0\equiv 1$ and we have excluded the case $\ell = k = 0$, it is thus sufficient to prove that $C_{\ell,k} \frac{\di}{\di r}F_k(1) > 0$ for $\ell \in \N_0$ and $k\in\N$. 
	
	The identity \cite[§15.5(i) Eq.~15.5.1]{NIST}
	\begin{equation*}
		\frac{\di}{\di z}\F(a,b;c;z) = \frac{ab}{c}\F(a+1,b+1;c+1;z)
	\end{equation*}
	implies in our context that
	\begin{equation*}
		C_{\ell,k} \frac{\di}{\di r}F_k(1) =  \frac{-2k(\ell+k+\tfrac{d}{2})}{\ell+\tfrac{d}{2}}C_{\ell,k}\frac{(-k)_{k-1}}{(\ell+1+\tfrac{d}{2})_{k-1}},
	\end{equation*}
	where the Chu--Vandermonde identity \cite[§15.4(ii) Eq.~15.4.24]{NIST} was used once again. Note that the term $(-k)_{k-1}$ contributes with the sign $(-1)^{k-1}$, which combined with the leading $-1$ and the sign $(-1)^k$ from $C_{\ell,k}$ proves the claim.
	
	\subsection*{(iv)}  
	
	We have
	\begin{equation*}
		\max_{r\in[0,1]}\abs{R_{\ell,k}(r)} \leq \max_{r\in[0,1]}\abs{C_{\ell,k}F_k(r)}.
	\end{equation*}
	According to the definition of $F_k$, it can alternatively be written as a Jacobi polynomial \cite[§18.5(iii) Eq.~18.5.7]{NIST}, 
	\begin{equation} \label{eq:Jacobi}
		F_k(r) = \widetilde{C} P^{(\ell + d/2 -1,0)}_k(1-2r^2),
	\end{equation}
	where $\widetilde{C}$ is an inconsequential constant. By \eqref{eq:Jacobi} and \cite[Theorem~3.24]{Shen_2011}, the maximum of $\abs{F_k(r)}$ is thus attained at either $r=0$ or $r=1$, for which we have:
	\begin{equation*}
		\abs{C_{\ell,k}F_k(0)} = \abs{C_{\ell,k}} \geq R_{\ell,k}(1) = \abs{C_{\ell,k}F_k(1)}. 
	\end{equation*}
	
	\subsection*{(v)}  
	
	It suffices to prove $\inner{r^{\ell+2p},R_{\ell,k}} = \chi_{\ell,k,p}$ for $\ell,p\in\N_0$ and $k\in\{0,\dots,p\}$ since we know from the proof of part (i) that
	\begin{equation*}
		r^{\ell+2p} = \sum_{k=0}^p \inner{r^{\ell+2p},R_{\ell,k}}R_{\ell,k}(r).
	\end{equation*}
	For $k=0$,
	\begin{equation*}
		\inner{r^{\ell+2p},R_{\ell,0}} = \inner{r^{\ell+2p},C_{\ell,0}r^\ell} = \frac{C_{\ell,0}}{2p+2\ell+d} = \chi_{\ell,0,p}.
	\end{equation*}
	Since this covers all the cases when $p=0$, we may assume $p\in\N$ in the following and prove the remaining cases using induction in $k$. Hence for $k\in\{1,\dots,p\}$, assume the induction hypothesis
	\begin{equation*}
		C_{\ell,k-1}\inner{r^{\ell+2p},r^\ell F_{k-1}} = \inner{r^{\ell+2p},R_{\ell,k-1}} = \chi_{\ell,k-1,p}.
	\end{equation*}
	Define
	\begin{equation*}
		c_j = \inner{r^{\ell+2p},r^\ell F_j} \quad \text{and} \quad \tilde{c}_j = \inner{r^{\ell+2p},r^\ell \widetilde{F}_j},
	\end{equation*}
	meaning that
	\begin{equation} \label{eq:rsumF}
		r^{\ell+2p} = \sum_{j=0}^p C_{\ell,j}^2c_j r^\ell F_j(r).
	\end{equation}
	By using integration by parts and \eqref{eq:diffFk}, 
	\begin{align*}
		dc_j &= d\int_0^1 \Bigl(r^{2\ell+2p}F_j\Bigr)r^{d-1}\,\di r \\
		&= F_j(1) - (2p+2\ell)\int_0^1 \Bigl(r^{2\ell+2p-1}F_j\Bigr)r^{d}\,\di r - \int_0^1 \Bigl(r^{2\ell+2p}\frac{\di}{\di r}F_j\Bigr)r^{d}\,\di r \\
		&= F_j(1) - (2p+2\ell)c_j + 2j\tilde{c}_j-2jc_j,
	\end{align*}
	which can be rearranged as
	\begin{equation} \label{eq:Finduct}
		(2p+2\ell+2j+d)c_j = F_j(1) + 2j\tilde{c}_j. 
	\end{equation}
	We will return to \eqref{eq:Finduct} a bit later. 
	
	Meanwhile, using \eqref{eq:rsumF} and \eqref{eq:Fjip} in the definition of $\tilde{c}_k$ gives
	\begin{equation*} 
		\tilde{c}_k = \sum_{j=0}^{k-1} C_{\ell,j}^2c_j\inner{r^{\ell}F_j,r^\ell\widetilde{F}_k}.
	\end{equation*}
	Formulas \eqref{eq:Fjident}, \eqref{eq:Fjip}, and the orthonormality of $R_{\ell,j}=C_{\ell,j}r^\ell F_j$ yield 
	\begin{align*}
		\tilde{c}_k &= \frac{\ell+2k+\frac{d-2}{2}-1}{\ell+k+\frac{d-2}{2}}c_{k-1} - \frac{k-1}{\ell+k+\frac{d-2}{2}}\sum_{j=0}^{k-2}C_{\ell,j}^2c_j\inner{r^\ell F_j,r^\ell\widetilde{F}_{k-1}} \\
		&= \frac{\ell+2k+\frac{d-2}{2}-1}{\ell+k+\frac{d-2}{2}}c_{k-1} - \frac{k-1}{\ell+k+\frac{d-2}{2}}\tilde{c}_{k-1}.
	\end{align*}
	Inserting the expression for $\tilde{c}_{k-1}$ from \eqref{eq:Finduct} results in
	\begin{align*}
		\tilde{c}_k &= \frac{(k-p-1)c_{k-1}+\frac{1}{2}F_{k-1}(1)}{\ell+k+\frac{d-2}{2}},
	\end{align*}
	which the induction hypothesis further transforms into
	\begin{align}
		\label{eq:formulax}
		2C_{\ell,k-1}\tilde{c}_k &= \frac{(2k-2p-2)\chi_{\ell,k-1,p}+R_{\ell,k-1}(1)}{\ell+k+\frac{d-2}{2}}.
	\end{align}

	Finally we return to \eqref{eq:Finduct} with $j=k$, multiply by $C_{\ell,k}/R_{\ell,k}(1)$, and employ \eqref{eq:formulax}:
	\begin{align}
		\frac{2p+2\ell+2k+d}{R_{\ell,k}(1)}\inner{r^{\ell+2p},R_{\ell,k}} &= 1 + \frac{2kC_{\ell,k}}{R_{\ell,k}(1)}\tilde{c}_k \notag\\ 
		&= 1 + \frac{kC_{\ell,k}}{R_{\ell,k}(1)C_{\ell,k-1}}\Bigl(\frac{(2k-2p-2)\chi_{\ell,k-1,p}+R_{\ell,k-1}(1)}{\ell+k+\frac{d-2}{2}}\Bigr) \notag\\
		&= 1+A(B-1). \label{eq:monomialfinal}
	\end{align}
	Due to the expressions for $C_{\ell,k}$ and $C_{\ell,k-1}$ from part (ii) and the definition of $\chi_{\ell,k-1,p}$, $A$ and $B$ can be written as:
	\begin{align*}
		A &= \frac{k}{\ell+k+\frac{d-2}{2}}\binom{\ell+k+\frac{d-2}{2}}{k}\binom{\ell+k-1+\frac{d-2}{2}}{k-1}^{-1} = 1, \\
		B &= \frac{p+1-k}{p+\ell+k+\frac{d-2}{2}}\binom{p}{k-1}\binom{p+\ell+k-1+\frac{d-2}{2}}{k-1}^{-1} = \binom{p}{k}\binom{p+\ell+k+\frac{d-2}{2}}{k}^{-1}.
	\end{align*}
	Hence, \eqref{eq:monomialfinal} reduces to $\inner{r^{\ell+2p},R_{\ell,k}} = \chi_{\ell,k,p}$.
	
	\subsection*{Acknowledgements}
	
	We thank Matteo Santacesaria (University of Genoa) for early discussions related to the subject of the paper in connection with Corollary~\ref{cor1}.
	
	HG is supported by grant 10.46540/3120-00003B from Independent Research Fund Denmark. MH and NH are supported by the Research Council of Finland (decisions 353081 and 359181).
	
	\appendix
	
	\section{Bounds on spherical harmonics} \label{sec:appA}
	
	We prove Proposition~\ref{prop:sphericalhambnds}, that for $f\in\mathcal{H}_\ell$ there is the sharp bound
	\begin{equation*}
		\norm{f}_{L^\infty(\partial B)} \leq \Bigl(\frac{\dim\mathcal{H}_\ell}{\abs{\partial B}}\Bigr)^{1/2}\norm{f}_{L^2(\partial B)}.
	\end{equation*}
	The result can also be found in lecture notes by Paul Garrett, University of Minnesota. However since we have no reference to the result in a permanent archive, we give our adaptation of the proof here for the sake of completion.
	
	Let $\inner{\,\cdot\,,\,\cdot\,}$ be the $L^2(\partial B)$ inner product. Since $\mathcal{H}_\ell$ is finite-dimensional and comprised of bounded continuous functions, the norms $\norm{\,\cdot\,}_{L^\infty(\partial B)}$ and $\norm{\,\cdot\,}_{L^2(\partial B)}$ are equivalent on $\mathcal{H}_\ell$. Hence, for a fixed $x\in\partial B$, the mapping $f\mapsto f(x)$ is a linear and bounded functional on $\mathcal{H}_\ell$ in the topology of $L^2(\partial B)$. By Riesz' representation theorem, for each $x\in\partial B$ there exists a unique $v_x\in \mathcal{H}_\ell$ such that
	\begin{equation*}
		f(x) = \inner{f,v_x}, \quad f\in \mathcal{H}_\ell.
	\end{equation*}
	Note that $\Delta_{\partial B}$ commutes with orthogonal coordinate transformations. So for any orthogonal matrix $U\in\R^{d\times d}$, the mapping $f\mapsto f\circ U$ is unitary on $\mathcal{H}_\ell$ with adjoint $f\mapsto f\circ U^{-1}$. Hence,
	\begin{equation*}
		\inner{f,v_{Ux}} = f\circ U(x) = \inner{f\circ U,v_x} = \inner{f,v_x\circ U^{-1}}, \quad f\in \mathcal{H}_\ell,
	\end{equation*}
	which gives the identification
	\begin{equation} \label{eq:vunitary0} 
		v_{Ux} = v_x\circ U^{-1}.
	\end{equation}
	In particular,
	\begin{equation} \label{eq:vunitary}
		v_x(z) = \inner{v_x,v_z} = \inner{v_x\circ U^{-1},v_z\circ U^{-1}} = \inner{v_{Ux},v_{Uz}} = v_{Ux}(Uz),
	\end{equation}
	which demonstrates that $x\mapsto v_x(x)$ is a constant function that we denote as
	\begin{equation*}
		V = v_x(x), \quad x\in\partial B.
	\end{equation*}
	Indeed, for any $x,y\in\partial B$, choosing in \eqref{eq:vunitary} $z = x$ and an orthogonal matrix $U$ such that $y = Ux$ gives $v_x(x) = v_y(y)$. An example of such a matrix is
	\begin{equation*}
		U = \begin{pmatrix}
			y & \widetilde{y}_2 & \cdots & \widetilde{y}_d
		\end{pmatrix}
		\begin{pmatrix}
			x & \widetilde{x}_2 & \cdots & \widetilde{x}_d
		\end{pmatrix}^\textup{T},
	\end{equation*}
	where $\{\widetilde{y}_i\}_{i=2}^d$ is an orthonormal basis for $\spanm\{y\}^\perp$, and $\{\widetilde{x}_i\}_{i=2}^d$ is an orthonormal basis for $\spanm\{x\}^\perp$.
	
	Next we will show that  
	\begin{equation} \label{eq:Veq}
		V = \norm{v_x}_{L^2(\partial B)}^2 = \norm{v_x}_{L^\infty(\partial B)}
	\end{equation}
	for any $x\in\partial B$. The first equality follows from the definition of $v_x$ by setting $f = v_x$. To deduce the second equality, we write 
	\begin{align}
		\abs{v_x(y)} &= \abs{\inner{v_x,v_y}} \leq \norm{v_x}_{L^2(\partial B)}\norm{v_y}_{L^2(\partial B)} \notag\\
		&= \norm{v_x}_{L^2(\partial B)}\norm{v_x\circ U^{-1}}_{L^2(\partial B)} = \norm{v_x}_{L^2(\partial B)}^2 = V, \label{eq:normvx}
	\end{align}
	where we again used an orthogonal matrix $U$ satisfying $y = Ux$ together with \eqref{eq:vunitary0}. As the upper bound in \eqref{eq:normvx} is attained        when $y=x$, we have proved \eqref{eq:Veq}.
	
	Finally, we expand $v_x$ in the finite orthonormal basis $\{f_{\ell,m}\}_{m\in\mathcal{I}_\ell}$,
	\begin{equation*}
		v_x = \sum_{m\in\mathcal{I}_\ell}\inner{v_x,f_{\ell,m}}f_{\ell,m} = \sum_{m\in\mathcal{I}_\ell}\overline{f_{\ell,m}(x)}f_{\ell,m},
	\end{equation*}
	which evaluated at $x$ gives
	\begin{equation*}
		V = \sum_{m\in \mathcal{I}_\ell}\abs{f_{\ell,m}(x)}^2
	\end{equation*}
	for any $x\in\partial B$. Integrating this expression results in
	\begin{equation} \label{eq:Vdim}
		V\abs{\partial B} = \sum_{m\in \mathcal{I}_\ell}\norm{f_{\ell,m}}_{L^2(\partial B)}^2 = \dim\mathcal{H}_\ell.
	\end{equation}
	Using \eqref{eq:Veq} and \eqref{eq:Vdim} proves the proposition statement,
	\begin{align*}
		\norm{f}_{L^\infty(\partial B)} &= \sup_{x\in\partial B}\abs{\inner{f,v_x}} \leq \sup_{x\in\partial B}\norm{v_x}_{L^2(\partial B)}\norm{f}_{L^2(\partial B)} \\
		&= V^{1/2}\norm{f}_{L^2(\partial B)} = \Bigl(\frac{\dim\mathcal{H}_\ell}{\abs{\partial B}}\Bigr)^{1/2}\norm{f}_{L^2(\partial B)}.
	\end{align*}
	Moreover, \eqref{eq:Veq} shows that the bound is sharp by choosing $f = v_x$ for any $x\in\partial B$.
	
	\bibliographystyle{plain}

\end{document}